\documentclass[11pt,reqno]{amsart}
\usepackage{amsmath,amssymb}

 \makeatletter
 \evensidemargin  \oddsidemargin
 \makeatother

\begin{document}
\thispagestyle{empty}
 \setcounter{page}{1}

\begin{center}
{\large\bf On the stability of generalized mixed type quadratic and
quartic functional equation in quasi-Banach spaces

\vskip.20in

{\bf M. Eshaghi Gordji } \\[2mm]

{\footnotesize Department of Mathematics,
Semnan University,\\ P. O. Box 35195-363, Semnan, Iran\\
[-1mm] e-mail: {\tt madjid.eshaghi@gmail.com}}

{\bf S. Abbaszadeh  } \\[2mm]

{\footnotesize Department of Mathematics,
Semnan University,\\ P. O. Box 35195-363, Semnan, Iran\\
[-1mm] e-mail: {\tt s.abbaszadeh.math@gmail.com}}}
\end{center}
\vskip 5mm

 \noindent{\footnotesize{\bf Abstract.}
In this paper, we establish the general solution of the functional
equation
$$f(nx+y)+f(nx-y)=n^2f(x+y)+n^2f(x-y)+2(f(nx)-n^2f(x))-2(n^2-1)f(y)\eqno \hspace {0 cm}$$for fixed integers $n$ with $n\neq0,\pm1$ and
investigate the generalized
  Hyers-Ulam-Rassias stability of this equation in quasi-Banach
  spaces.
 \vskip.10in

  \newtheorem{df}{Definition}[section]
  \newtheorem{rk}[df]{Remark}
   \newtheorem{lem}[df]{Lemma}
   \newtheorem{thm}[df]{Theorem}
   \newtheorem{pro}[df]{Proposition}
   \newtheorem{cor}[df]{Corollary}
   \newtheorem{ex}[df]{Example}

 \setcounter{section}{0}
 \numberwithin{equation}{section}

\vskip .2in

\begin{center}
\section{Introduction}
\end{center}

The stability problem of functional equations originated from a
question of Ulam [21] in 1940, concerning the stability of group
homomorphisms. Let $(G_1,.)$ be a group and let $(G_2,*)$ be a
metric group with the metric $d(.,.).$ Given $\epsilon >0$, dose
there exist a $\delta
>0$, such that if a mapping $h:G_1\longrightarrow G_2$ satisfies the
inequality $d(h(x.y),h(x)*h(y)) <\delta$ for all $x,y\in G_1$, then
there exists a homomorphism $H:G_1\longrightarrow G_2$ with
$d(h(x),H(x))<\epsilon$ for all $x\in G_1?$ In the other words,
under what condition dose there exist a homomorphism near an
approximate homomorphism? The concept of stability for functional
equation arises when we replace the functional equation by an
inequality which acts as a perturbation of the equation.
 In 1941, D.H. Hyers [10] gave a first affirmative  answer to the question of
Ulam for Banach spaces. Let $f:{E}\longrightarrow{E'}$ be a mapping
between Banach spaces such that
$$\|f(x+y)-f(x)-f(y)\|\leq \delta $$
for all $x,y\in E,$ and for some $\delta>0.$ Then there exists a
unique additive mapping $T:{E}\longrightarrow{E'}$ such that
$$\|f(x)-T(x)\|\leq \delta$$
for all $x\in E.$ Moreover, if $f(tx)$ is continuous in t for each
fixed $x\in E,$ then $T$ is linear. In 1978, Th. M. Rassias [17]
provided a generalization of Hyers' Theorem which allows the Cauchy
difference to be unbounded. The functional equation
$$f(x+y)+f(x-y)=2f(x)+2f(y),\eqno \hspace {0.5 cm}(1.1)$$
is related to symmetric bi-additive function[1,2,11,13]. It is
natural that this equation is called a quadratic functional
equation. In particular, every solution of the quadratic equation
(1.1) is said to be a quadratic function. It is well known that a
function $f$ between real vector spaces is quadratic if and only if
there exits a unique symmetric bi-additive function $B$ such that
$f(x)=B(x,x)$ for all $x$ (see [1,13]). The bi-additive function $B$
is given by
$$B(x,y)=\frac{1}{4}(f(x+y)-f(x-y)).\eqno \hspace {0.5 cm}(1.2)$$
A Hyers-Ulam-Rassias stability problem for the quadratic functional
equation (1.1) was proved by Skof for functions $f:A\longrightarrow
B$, where A is normed space and B Banach space (see [19]). Cholewa
[4] noticed that the Theorem of Skof is still true if relevant
domain $A$ is replaced an abelian group. In the paper [6] , Czerwik
proved the Hyers-Ulam-Rassias stability of the equation (1.1).
Grabiec [9] has generalized these result mentioned above.

In [14], Won-Gil Prak and Jea Hyeong Bae, considered the following
quartic functional equation:
$$f(x+2y)+f(x-2y)=4(f(x+y)+f(x-y)+6f(y))-6f(x).\eqno\hspace {2.9cm}(1.3)$$
In fact, they proved that a function
 $f$ between two real vector spaces $X$ and $Y$ is a solution of (1.3) if and only if there
 exists a unique symmetric multi-additive function $D:X\times X\times X\times X\longrightarrow Y$ such that
 $f(x)=D(x,x,x,x)$ for all $x$. It is easy to show that
 the function $f(x)=x^4$ satisfies the functional equation (1.4), which is called
a quartic functional equation (see also [5]).

In addition H. Kim [12], has obtained the generalized
Hyers-Ulam-Rassias stability for the following mixed type of quartic
and quadratic functional equation:
$$\biguplus^{n-1}_{x_{2},...,x_{n}}f(x_{1})+2^{n-1}(n-2)\sum^{n}_{i=1}f(x_{i})=2^{n-2}\sum_{1\leq i<j\leq n}[\biguplus_{x_{j}}f(x_{i})]\eqno\hspace {2.9cm}(1.4)$$
for all n-variables $x_{1},x_{2},...,x_{n}\in E_{1}$, where $n>2$
and $f:{E_{1}}\longrightarrow{E_{2}}$ be a function between two
real linear spaces $E_{1}$ and $E_{2}$.

 Also A. Najati
and G. Zamani Eskandani [16], have established the general solution
and the generalized Hyers-Ulam-Rassias stability for a mixed type of
cubic and additive functional equation, whenever $f$ is a mapping
between two quasi-Banach spaces.

 Now, we introduce the following
functional equation for fixed integers $n$ with $n\neq0,\pm1$:
\begin{align*}
f(nx+y)+f(nx-y)&=n^2f(x+y)+n^2f(x-y)+2f(nx)\\&-2n^2f(x)-2(n^2-1)f(y)
\hspace{4.9cm} (1.5)\end{align*} in quasi Banach spaces. It is
easy to see that the function $f(x)=ax^4+bx^2$ is a solution of
the functional equation (1.5). In the present paper we investigate
the general solution of functional equation (1.5) when $f$ is a
function between vector spaces, and we establish the generalized
Hyers-Ulam-Rassias stability of this functional equation whenever
$f$ is a function between two quasi-Banach spaces.

 We recall some basic facts concerning quasi-Banach
space and some preliminary results.
\begin{df}\label{t2} (See [3, 18].) Let $X$ be a real linear space.
A quasi-norm is a real-valued function on $X$ satisfying the
following:\\ (1) $\|x\| \geq 0$ for all $x\in X$ and $\|x\|=0$ if
and only if $x=0~.$\\
(2) $\|\lambda.x\|=|\lambda|.\|x\|$ for all $\lambda \in \Bbb R$
and all $x \in X~.$\\
(3) There is a constant $K \geq 1$ such that $\|x+y\| \leq
K(\|x\|+\|y\|)$ for all $x,y \in X~.$

It follows from condition (3) that
$$\|\sum^{2m}_{i=1} x_{i}\| \leq M^m \sum^{2m}_{i=1}\|x_{i}\|,\hspace {1.5 cm}
\|\sum^{2m+1}_{i=1} x_{i}\| \leq M^{m+1}
\sum^{2m+1}_{i=1}\|x_{i}\|$$for all $~m\geq 1$ and all
$~~x_{1},x_{2},....,x_{2m+1}\in X.$
\end{df}
The pair $(X,\|.\|)$ is called a quasi-normed space if $\|.\|$ is a
quasi-norm on $X~.$ The smallest possible $M$ is called the modulus
of concavity of $\|.\|.$ A quasi-Banach space is a complete
quasi-normed space.

 A quasi-norm $\|.\|$ is called a
p-norm $(0 < p \leq 1)$ if
$$\|x+y\|^p \leq \|x\|^p+\|y\|^p$$
for all $x,y \in X~.$ In this case, a quasi-Banach space is called a
p-Banach space.

Given a p-norm, the formula $d(x,y):=\|x-y\|^p$ gives us a
translation invariant metric on X. By the Aoki-Rolewicz Theorem [
18](see also [3]), each quasi-norm is equivalent to some p-norm.
Since it is much easier to work with p-norms, henceforth we restrict
our attention mainly to p-norms. In [20], J. Tabor has investigated
a version of Hyers-Rassias-Gajda theorem (see[7,17]) in quasi-Banach
spaces.
\\  \vskip .2in

\section{ General solution}
Throughout this section, $X$ and $Y$  will be  real vector spaces.
We here present the general solution of (1.5).
\begin{lem}\label{t2} If a function $f:X\longrightarrow Y$ satisfies
the functional equation (1.5), then f is a quadratic and quartic
function.\end{lem}
\begin{proof}
By letting $x=y=0$ in (1.5),we get $f(0)=0$.Set $x=0$ in (1.5) to
get $f(y)=f(-y)$ for all $y \in X$.So the function $f$ is even.We
substitute  $x=x+y$ in (1.5) and then $x=x-y$ in (1.5) to obtain
that
\begin{align*}
f(nx+(n+1)y)+f(nx+(n-1)y)&=n^2f(x+2y)+n^2f(x)+2f(nx+ny)\\&-2n^2f(x+y)-2(n^2-1)f(y)
\hspace{2.3cm} (2.1)\end{align*} and
\begin{align*}
f(nx-(n-1)y)+f(nx-(n+1)y)&=n^2f(x)+n^2f(x-2y)+2f(nx-ny)\\&-2n^2f(x-y)-2(n^2-1)f(y)
\hspace{2.3cm} (2.2)\end{align*} for all $x,y \in X$.Interchanging x
and y in (1.5) and using evenness of $f$ to get the relation
\begin{align*}
f(x+ny)+f(x-ny)&=n^2f(x+y)+n^2f(x-y)+2f(ny)\\&-2n^2f(y)-2(n^2-1)f(x)
\hspace{4.9cm} (2.3)\end{align*} for all $x,y \in X$.Replacing $y$
by $ny$ in (1.5) and then using (2.3),we have
\begin{align*}
f(nx+ny)+f(nx-ny)&=n^4f(x+y)+n^4f(x-y)+2f(ny)\\&+2f(nx)-2n^4f(x)-2n^4f(y)
\hspace{3.9cm} (2.4)\end{align*} for all $x,y \in X$.If we add (2.1)
to (2.2) and use (2.4),we have
\begin{align*}
f&(nx+(n+1)y)+f(nx-(n+1)y)+f(nx+(n-1)y)+f(nx-(n-1)y)=\\&n^2f(x+2y)+n^2f(x-2y)+2n^2(n^2-1)f(x+y)
+2n^2(n^2-1)f(x-y)\\&+4f(ny)+4f(nx)+(-4n^4+2n^2)f(x)+(-4n^4-4n^2+4)f(y)
\hspace{2.6cm} (2.5)\end{align*}for all $x,y \in X$.Substitute
$y=x+y$ in (1.5) and then $y=x-y$ in (1.5) and using evenness of $f$
to obtain that
\begin{align*}
f((n+1)x+y)+f((n-1)x-y)&=n^2f(2x+y)+n^2f(y)+2f(nx)\\&-2n^2f(x)-2(n^2-1)f(x+y)
\hspace{2.9cm} (2.6)\end{align*} and
\begin{align*}
f((n+1)x-y)+f((n-1)x+y)&=n^2f(2x-y)+n^2f(y)+2f(nx)\\&-2n^2f(x)-2(n^2-1)f(x-y)
\hspace{2.8cm} (2.7)\end{align*} for all $x,y \in X$.Interchanging x
with y in (2.6) and (2.7) and using evenness of $f$,we get the
relations
\begin{align*}
f(x+(n+1)y)+f(x-(n-1)y)&=n^2f(x+2y)+n^2f(x)+2f(ny)\\&-2n^2f(y)-2(n^2-1)f(x+y)
\hspace{2.9cm} (2.8)\end{align*} and
\begin{align*}
f(x-(n+1)y)+f(x+(n-1)y)&=n^2f(x-2y)+n^2f(x)+2f(ny)\\&-2n^2f(y)-2(n^2-1)f(x-y)
\hspace{2.9cm} (2.9)\end{align*}for all $x,y \in X$.With the
substitution $y=(n+1)y$ in (1.5) and then $y=(n-1)y$ in (1.5), we
have
\begin{align*}
f(nx+(n+1)y)+f(nx-(n+1)y)&=n^2f(x+(n+1)y)+n^2f(x-(n+1)y)+2f(nx)\\&-2n^2f(x)-2(n^2-1)f((n+1)y)
\hspace{2.1cm} (2.10)\end{align*} and
\begin{align*}
f(nx+(n-1)y)+f(nx-(n-1)y)&=n^2f(x+(n-1)y)+n^2f(x-(n-1)y)+2f(nx)\\&-2n^2f(x)-2(n^2-1)f((n-1)y)
\hspace{2.1cm} (2.11)\end{align*}for all $x,y \in X$.Replacing $x$
by $y$ in (1.5), we obtain
$$f((n+1)y)+f((n-1)y)=n^2f(2y)-2(2n^2-1)f(y)+2f(ny)\eqno \hspace {3.5 cm} (2.12)$$for all $y \in
X$.Adding (2.10) with (2.11) and using (2.8), (2.9) and (2.12), we
lead to
\begin{align*}
f&(nx+(n+1)y)+f(nx-(n+1)y)+f(nx+(n-1)y)+f(nx-(n-1)y)=\\&n^4f(x+2y)+n^4f(x-2y)-2n^2(n^2-1)f(x+y)
-2n^2(n^2-1)f(x-y)\\&
+4f(ny)+4f(nx)-2n^2(n^2-1)f(2y)+(2n^4-4n^2)f(x)\\&+(4n^4-12n^2+4)f(y)
\hspace{8.6cm} (2.13)\end{align*}for all $x,y \in X$.By comparing
(2.5) with (2.13), we arrive at
$$f(x+2y)+f(x-2y)=4f(x+y)+4f(x-y)+2f(2y)-8f(y)-6f(x)\eqno \hspace {1.9cm} (2.14)$$for all $x,y \in
X$.Interchange $x$ with $y$ in (2.14) and use evenness of $f$ to get
the relation
$$f(2x+y)+f(2x-y)=4f(x+y)+4f(x-y)+2f(2x)-8f(x)-6f(y)\eqno \hspace {1.9cm} (2.15)$$for all $x,y \in
X$.

We would show that (2.15) is a quadratic and quartic functional
equation.To get this, we show that the functions $g:X\longrightarrow
Y$ defined by $g(x)=f(2x)-16f(x)$ for all $x \in X$ and
$h:X\longrightarrow Y$ defined by $h(x)=f(2x)-4f(x)$ for all $x \in
X$, are quadratic and quartic, respectively.

Replacing $y$ by $2y$ in (2.15) and using evenness of $f$,we have
$$f(2x+2y)+f(2x-2y)=4f(2y+x)+4f(2y-x)+2f(2x)-8f(x)-6f(2y)\eqno \hspace {0.6cm} (2.16)$$for all $x,y \in
X$. By interchanging $x$ with $y$ in (2.16) and then using
(2.15),we obtain by evenness of $f$
\begin{align*}
f(2x+2y)+f(2x-2y)&=4f(2x+y)+4f(2x-y)+2f(2y)-8f(y)-6f(2x)\\&=16f(x+y)+16f(x-y)+2f(2x)+2f(2y)\\&-32f(x)
-32f(y)\hspace{6.0cm} (2.17)\end{align*}for all $x,y \in X$.By
rearranging (2.17), we have
\begin{align*}
[&f(2x+2y)-16f(x+y)]+[f(2x-2y)-16f(x-y)]=\\&2[f(2x)-16f(x)]+2[f(2y)-16f(y)]\hspace{7.7cm}\end{align*}for
all $x,y \in X$.This means that
$$g(x+y)+g(x-y)=2g(x)=2g(y)\eqno \hspace {0cm}$$for
all $x,y \in X$.Therefore the function $g:X\longrightarrow Y$ is
quadratic.

To prove that $h:X\longrightarrow Y$ is quartic,we have to show that
$$h(2x+y)+h(2x-y)=4h(x+y)+4h(x-y)+24h(x)-6h(y)\eqno \hspace {0cm}$$for
all $x,y \in X$.Replacing $x$ and $y$ by $2x$ and $2y$ in (2.15),
respectively,we get
$$f(4x+2y)+f(4x-2y)=4f(2x+2y)+4f(2x-2y)+2f(4x)-8f(2x)-6f(2y)\eqno \hspace {1.0cm} (2.18)$$for all $x,y \in
X$.Since $g(2x)=4g(x)$ for all $x \in X$ where $g:X\longrightarrow
Y$ is a quadratic function defined above, we have
$$f(4x)=20f(2x)-64f(x)\eqno \hspace {8.0cm} (2.19)$$for all $x \in
X$.Hence, it follows from (2.15) ,(2.18) and (2.19) that
\begin{align*}
h(2x+y)+h(2x-y)&=[f(4x+2y)-4f(2x+y)]+[f(4x-2y)-4f(2x-y)]\\&=4[f(2x+2y)-4f(x+y)]+4[f(2x-2y)-4f(x-y)]\\&
+24[f(2x)-4f(x)]-6[f(2y)-4f(y)]\\&=4h(x+y)+4h(x-y)+24h(x)-6h(y)\hspace{5.7cm}\end{align*}for
all $x,y \in X$.Therefore, $h:X\longrightarrow Y$ is a quartic
function.
\end{proof}

\begin{thm}\label{t2} A function $f:X\rightarrow Y$ satisfies
(1.5) if and only if there exist a unique symmetric multi-additive
function $D:X\times X\times X\times X\longrightarrow Y$ and a unique
symmetric bi-additive function $B:X\times X\longrightarrow Y$  such
that
$$f(x)=D(x,x,x,x)+B(x,x)\hspace{6cm}$$ for all $x\in X.$
\end{thm}

\begin{proof}We first assume that the function $f:X\longrightarrow
Y$ satisfies (1.5).Let $g, h :X\rightarrow Y$ be functions defined
by
$$ g(x):=f(2x)-16f(x)\hspace{2cm} h(x):=f(2x)-4f(x)\hspace{2cm}$$
 for all $x\in X.$ Hence, by Lemma (2.1), we achieve that the
 functions $g$ and $h$ are quadratic and quartic, respectively, and
$$ f(x):=\frac{1}{12}h(x)-\frac{1}{12}g(x)\hspace{7cm}$$
for all $x\in X.$ Therefore, there exist a unique symmetric
multi-additive mapping $D:X\times X\times X\times X\longrightarrow
Y$ and a unique symmetric bi-additive mapping $B:X\times
X\longrightarrow Y$  such that $D(x,x,x,x)=\frac{1}{12}h(x)$ and
$B(x,x)=-\frac{1}{12}g(x)$ for all $x\in X$(see[1, 14]). So
$$f(x)=D(x,x,x,x)+B(x,x)\hspace{6cm}$$ for all $x\in X.$

Conversely assume that
$$f(x)=D(x,x,x,x)+B(x,x)\hspace{6cm}$$ for all $x\in X,$ where the
function $D:X\times X\times X\times X\longrightarrow Y$ is symmetric
multi-additive and $B:X\times X\longrightarrow Y$ is bi-additive
defined above.By a simple computation, one can show that the
functions $D$ and $B$ satisfy the functional equation (1.5),so the
function f satisfies (1.5).
\end{proof}
\section{Hyers-Ulam-Rassias stability of Eq.(1.5)}
From now on, let $X$ and $Y$ be a quasi-Banach space with quasi-norm
$\|.\|_{X}$ and a p-Banach space  with p-norm
$\|.\|_{Y}$,respectively.Let $M$ be the modulus of concavity of
$\|.\|_{Y}$.In this section using an idea of G$\check{a}$vruta[8] we
prove the stability of Eq.(1.5) in the spirit of Hyers, Ulam and
Rassias.For convenience we use the following abbreviation for a
given function $f:X\longrightarrow Y$:
$$\bigtriangleup f(x,y)=f(nx+y)+f(nx-y)-n^2f(x+y)-n^2f(x-y)-2f(nx)+2n^2f(x)+2(n^2-1)f(y)\eqno \hspace {0cm}$$
for all $x,y\in X$.We will use the following lemma in this section.
\begin{lem}\label{t'2}(see [15].) Let $0<p\leq1$ and let
$x_1,x_2,\ldots,x_n$ be non-negative real numbers. Then
$$(\sum^{n}_{i=1} x_i)^p\leq \sum^{n}_{i=1} {x_i}^p. $$
\end{lem}
\begin{thm}\label{t2}
Let $\varphi_q:X\times X\rightarrow [0,\infty)$ be a function such
that
$$\lim_{m\rightarrow\infty} 4^{m}
\varphi_q(\frac{x}{2^{m}},\frac{y}{2^{m}})=0 \eqno\hspace
{2cm}(3.1)$$ for all $x,y\in X$ and
$$\sum^{\infty}_{i=1} 4^{pi}
{\varphi_q}^p(\frac{x}{2^{i}},\frac{y}{2^{i}})<\infty \eqno \hspace
{2cm}(3.2)$$ for all $x\in X$ and for all $y\in
\{x,2x,3x,nx,(n+1)x,(n-1)x,(n+2)x,(n-2)x,(n-3)x\}.$ Suppose that a
function $f:X\rightarrow Y$ with $f(0)=0$ satisfies the inequality
$$\|\bigtriangleup f(x,y)\|_Y \leq\varphi_q(x,y)\eqno \hspace {4.5cm} (3.3)$$
for all $x,y\in X.$ Then the limit
$$Q(x):=\lim_{m\rightarrow\infty} 4^{m}
[f(\frac{x}{2^{m-1}})-16f(\frac{x}{2^{m}})] \eqno \hspace
{2.4cm}(3.4)$$ exists for all $x\in X$ and $Q:X\rightarrow Y$ is a
unique quadratic function satisfying
$$\|f(2x)-16f(x)-Q(x)\|_Y \leq\frac{M^{11}}{4}[\widetilde{\psi}_q(x)]^\frac{1}{p}\eqno\hspace {2cm}(3.5)$$
for all $x\in X,$ where
\begin{align*}
\widetilde{\psi}_q(x):&=\sum^{\infty}_{i=1}
4^{pi}~\textbf{\{}\frac{1}{n^{2p}(n^2-1)^p}
~[~\varphi^p_q(\frac{x}{2^i},\frac{(n+2)x}{2^i})+\varphi^p_q(\frac{x}{2^i},\frac{(n-2)x}{2^i})\\&+4^p\varphi^p_q(\frac{x}{2^i},\frac{(n+1)x}{2^i})+4^p\varphi^p_q(\frac{x}{2^i},\frac{(n-1)x}{2^i})+4^p\varphi^p_q(\frac{x}{2^i},\frac{nx}{2^i})+\varphi^p_q(\frac{2x}{2^i},\frac{2x}{2^i})\\&+4^p\varphi^p_q(\frac{2x}{2^i},\frac{x}{2^i})+n^{2p}\varphi^p_q(\frac{x}{2^i},\frac{3x}{2^i})+2^p(3n^2-1)^p\varphi^p_q(\frac{x}{2^i},\frac{2x}{2^i})\\&+(17n^2-8)^p\varphi^p_q(\frac{x}{2^i},\frac{x}{2^i})+\frac{n^{2p}}{(n^2-1)^p}(\varphi^p_q(0,\frac{x(n+1)x}{2^i})+\varphi^p_q(0,\frac{(n-3)x}{2^i})\\&+10^p\varphi^p_q(0,\frac{(n-1)x}{2^i})+4^p\varphi^p_q(0,\frac{nx}{2^i})+4^p\varphi^p_q(0,\frac{(n-2)x}{2^i}))+\frac{(n^4+1)^p}{(n^2-1)^p}\varphi^p_q(0,\frac{2x}{2^i})\\&+\frac{(2(3n^4-n^2+2))^p}{(n^2-1)^p}\varphi^p_q(0,\frac{x}{2^i})~]~
\textbf{\}}. \hspace{6.4cm}(3.6)\end{align*}
\end{thm}
\begin{proof}
Set $x=0$ in (3.3) and then interchange $x$ with $y$ to get
$$\|(n^2-1)f(x)-(n^2-1)f(-x)\|\leq
\varphi_q(0,x) \eqno \hspace {6cm}(3.7)$$for all $x\in X$.Replacing
$y$ by $x$, $2x$, $nx$, $(n+1)x$ and $(n-1)x$ in
(3.3),respectively,we get
$$\|f((n+1)x)+f((n-1)x)-n^2f(2x)-2f(nx)+(4n^2-2)f(x)\|\leq
\varphi_q(x,x)\eqno \hspace {0.9cm}(3.8)$$and
\begin{align*}\|f((n+2)x)+f((n-2)x)-n^2f(3x)&-n^2f(-x)-2f(nx)+2n^2f(x)\\&+2(n^2-1)f(2x)\|\leq
\varphi_q(x,2x)\hspace {2.6cm}(3.9)\end{align*}and
\begin{align*}\|f(2nx)-n^2f((n+1)x)-n^2f((1-n)x)&+2(n^2-2)f(nx)+2n^2f(x)\|\\&\leq
\varphi_q(x,nx)\hspace {4.6cm}(3.10)\end{align*}and
\begin{align*}\|f((2n+1)x)+f(-x)&-n^2f((n+2)x)-n^2f(-nx)-2f(nx)+2n^2f(x)\\&+2(n^2-1)f((n+1)x)\|\leq
\varphi_q(x,(n+1)x)\hspace {2.9cm}(3.11)\end{align*}and
\begin{align*}\|f((2n-1)x)+f(x)&-n^2f((2-n)x)-(n^2+2)f(nx)+2n^2f(x)\\&+2(n^2-1)f((n-1)x)\|\leq
\varphi_q(x,(n-1)x)\hspace {3.2cm}(3.12)\end{align*}and
\begin{align*}\|f(2(n+1)x)+f(-2x)&-n^2f((n+3)x)-n^2f(-(n+1)x)-2f(nx)+2n^2f(x)\\&+2(n^2-1)f((n+2)x)\|\leq
\varphi_q(x,(n+2)x)\hspace {2.8cm}(3.13)\end{align*}and
\begin{align*}\|f((2(n-1)x)+f(2x)&-n^2f((n-1)x)-n^2f(-(n-3)x)-2f(nx)+2n^2f(x)\\&+2(n^2-1)f((n-2)x)\|\leq
\varphi_q(x,(n-2)x)\hspace {2.9cm}(3.14)\end{align*}and
\begin{align*}\|f((n+3)x)+f((n-3)x)&-n^2f(4x)-n^2f(-2x)-2f(nx)+2n^2f(x)\\&+2(n^2-1)f(3x)\|\leq
\varphi_q(x,3x)\hspace {4.0cm}(3.15)\end{align*} for all $x\in X$.We
combine (3.7) with (3.9), (3.10), (3.11), (3.12), (3.13), (3.14) and
(3.15), respectively, to get the following inequalities:
\begin{align*}\|f((n+2)x)+f((n-2)x)&-n^2f(3x)-n^2f(x)-2f(nx)+2n^2f(x)\\&+2(n^2-1)f(2x)\|\leq
\varphi_q(x,2x)+\frac{n^2}{n^2-1}\varphi_q(0,x)\hspace
{1.5cm}(3.16)\end{align*}and
\begin{align*}\|f(2nx)-n^2f((n+1)x)&-n^2f((n-1)x)+2(n^2-2)f(nx)+2n^2f(x)\|\\&\leq
\varphi_q(x,nx)+\frac{n^2}{n^2-1}\varphi_q(0,(n-1)x)\hspace
{3.4cm}(3.17)\end{align*}and
\begin{align*}\|&f((2n+1)x)+f(x)-n^2f((n+2)x)-n^2f(nx)-2f(nx)+2n^2f(x)\\&+2(n^2-1)f((n+1)x)\|\leq
\varphi_q(x,(n+1)x)\\&+\frac{n^2}{n^2-1}\varphi_q(0,nx)+\frac{1}{n^2-1}\varphi_q(0,x)\hspace
{6.9cm}(3.18)\end{align*}and
\begin{align*}\|&f((2n-1)x)+f(x)-n^2f((n-2)x)-(n^2+2)f(nx)+2n^2f(x)\\&+2(n^2-1)f((n-1)x)\|\leq
\varphi_q(x,(n-1)x)+\frac{n^2}{n^2-1}\varphi_q(0,(n-2)x)\hspace
{2.5cm}(3.19)\end{align*}and
\begin{align*}\|f(2(n+1)x)+f(2x)&-n^2f((n+3)x)-n^2f((n+1)x)-2f(nx)+2n^2f(x)\\&+2(n^2-1)f((n+2)x)\|\leq
\varphi_q(x,(n+2)x)\\&+\frac{n^2}{n^2-1}\varphi_q(0,(n+1)x)+\varphi_q(0,2x)\hspace
{3.9cm}(3.20)\end{align*}and
\begin{align*}\|f&(2(n-1)x)+f(2x)-n^2f((n-1)x)-n^2f((n-3)x)-2f(nx)+2n^2f(x)\\&+2(n^2-1)f((n-2)x)\|\leq
\varphi_q(x,(n-2)x)+\frac{n^2}{n^2-1}\varphi_q(0,(n-3)x)\hspace
{2.3cm}(3.21)\end{align*}and
\begin{align*}\|f((n+3)x)+f((n-3)x)&-n^2f(4x)-n^2f(2x)-2f(nx)+2n^2f(x)\\&+2(n^2-1)f(3x)\|\leq
\varphi_q(x,3x)+\frac{n^2}{n^2-1}\varphi_q(0,2x)\hspace
{1.4cm}(3.22)\end{align*}for all $x\in X$.Replacing $x$ and $y$ by
$2x$ and $x$ in (3.3), respectively, we obtain
\begin{align*}\|f((2n+1)x)+f((2n-1)x)-n^2f(3x)&-2f(2nx)+2n^2f(2x)\\&+(n^2-2)f(x)\|\leq
\varphi_q(2x,x)\hspace {2.5cm}(3.23)\end{align*}for all $x\in
X$.Putting $2x$ and $2y$ instead of $x$ and $y$ in (3.3),
respectively, we have
\begin{align*}\|f(2(n+1)x)+f(2(n-1)x)-n^2f(4x)-2f(2nx)&+2(2n^2-1)f(2x)\|\\&\leq
\varphi_q(2x,2x)\hspace {3.1cm}(3.24)\end{align*}for all $x\in X$.It
follows from (3.8), (3.16), (3.17), (3.18), (3.19) and (3.23) that
\begin{align*}\|&f(3x)-6f(2x)+15f(x)\|\leq
\frac{M^5}{n^{2}(n^2-1)}[\varphi_q(x,(n+1)x)+\varphi_q(x,(n-1)x)\\&+\varphi_q(2x,x)+2\varphi_q(x,nx)+n^2\varphi_q(x,2x)+(4n^2-2)\varphi_q(x,x)\\&+\frac{n^2}{n^2-1}(2\varphi_q(0,(n-1)x)+\varphi_q(0,nx)+\varphi_q(0,(n-2)x))\\&+\frac{n^4+1}{n^2-1}\varphi_q(0,x)]\hspace
{9.3cm}(3.25)\end{align*}for all $x\in X$.Also, from (3.8), (3.16),
(3.17), (3.20), (3.21), (3.22) and (3.24), we conclude
\begin{align*}\|f&(4x)-4f(3x)+4f(2x)+4f(x)\|\leq
\frac{M^6}{n^{2}(n^2-1)}[\varphi_q(x,(n+2)x)\\&+\varphi_q(x,(n-2)x)+\varphi_q(2x,2x)+2\varphi_q(x,nx)+n^2(\varphi_q(x,3x)+\varphi_q(x,x))\\&+2(n^2-1)\varphi_q(x,2x)+\frac{n^2}{n^2-1}(2\varphi_q(0,(n-1)x)+\varphi_q(0,(n-3)x)\\&+\varphi_q(0,(n+1)x))+\frac{n^4+1}{n^2-1}\varphi_q(0,2x)+2n^2\varphi_q(0,x)]\hspace
{4.3cm}(3.26)\end{align*}for all $x\in X$.Finally, combining (3.25)
and (3.26) yields
\begin{align*}\|f&(4x)-24f(2x)+64f(x)\|\leq
\frac{M^{8}}{n^{2}(n^2-1)}[\varphi_q(x,(n+2)x)+\varphi_q(x,(n-2)x)\\&+4\varphi_q(x,(n+1)x)+4\varphi_q(x,(n-1)x)+10\varphi_q(x,nx)+\varphi_q(2x,2x)\\&+4\varphi_q(2x,x)+n^2\varphi_q(x,3x)+2(3n^2-1)\varphi_q(x,2x)+(17n^2-8)\varphi_q(x,x)\\&+\frac{n^2}{n^2-1}(\varphi_q(0,(n+1)x)+\varphi_q(0,(n-3)x)+10\varphi_q(0,(n-1)x)+4\varphi_q(0,nx)\\&+4\varphi_q(0,(n-2)x))+\frac{n^4+1}{n^2-1}\varphi_q(0,2x)+\frac{2(3n^4-n^2+2)}{n^2-1}\varphi_q(0,x)]\hspace
{2.4cm}(3.27)\end{align*}for all $x\in X$.By substituting
\begin{align*}\psi_{q}(x)&=\frac{1}{n^{2}(n^2-1)}[\varphi_q(x,(n+2)x)+\varphi_q(x,(n-2)x)\\&+4\varphi_q(x,(n+1)x)+4\varphi_q(x,(n-1)x)+10\varphi_q(x,nx)+\varphi_q(2x,2x)\\&+4\varphi_q(2x,x)+n^2\varphi_q(x,3x)+2(3n^2-1)\varphi_q(x,2x)+(17n^2-8)\varphi_q(x,x)\\&+\frac{n^2}{n^2-1}(\varphi_q(0,(n+1)x)+\varphi_q(0,(n-3)x)+10\varphi_q(0,(n-1)x)+4\varphi_q(0,nx)\\&+4\varphi_q(0,(n-2)x))+\frac{n^4+1}{n^2-1}\varphi_q(0,2x)+\frac{2(3n^4-n^2+2)}{n^2-1}\varphi_q(0,x)]\hspace
{2.0cm}(3.28)\end{align*}(3.27) gives
$$\|f(4x)-20f(2x)+64f(x)\|\leq M^{8}\psi_{q}(x)\eqno \hspace {6.2cm}(3.29)$$
for all $x \in X.$

Let $g:X \to Y$ be a function defined by $g(x):=f(2x)-16f(x)$ for
all $x \in X.$From (3.29), we conclude that
$$\|g(2x)-4g(x)\|\leq M^{8}\psi_{q}(x) \eqno\hspace {7.8cm}(3.30)$$ for all $x \in X.$
If we replace $x$ in (3.30) by $\frac{x}{2^{m+1}}$ and multiply
both sides of (3.30) by $4^m,$ we get
$$\|4^{m+1}g(\frac{x}{2^{m+1}})-4^mg(\frac{x}{2^m})\|_Y\leq M^8 4^m \psi_{q}(\frac{x}{2^{m+1}})
 \eqno\hspace {3.1cm}(3.31)$$
for all $x\in X$ and all non-negative integers $m$. Since $Y$ is a
p-Banach space, then inequality (3.31) gives
\begin{align*}
\|4^{m+1}g(\frac{x}{2^{m+1}})-4^k g(\frac{x}{2^k})\|_Y^p&\leq
\sum^{m}_{i=k}\|4^{i+1}g(\frac{x}{2^{i+1}})-4^i
g(\frac{x}{2^i})\|_Y^p\\
&\leq M^{8p}\sum^{m}_{i=k} 4^{ip} {\psi_{q}}^p (\frac{x}{2^{i+1}})
\hspace {4.5cm}(3.32)
\end{align*}
for all non-negative integers $m$ and $k$ with $m\geq k$ and  for
all $x\in X.$ Since $0< p\leq1$, then by Lemma 3.1, from (3.28), we
conclude that
\begin{align*}\psi_{q}^p(x)&\leq\frac{1}{n^{2p}(n^2-1)^p}[\varphi^p_q(x,(n+2)x)+\varphi^p_q(x,(n-2)x)\\&+4^p\varphi^p_q(x,(n+1)x)+4^p\varphi^p_q(x,(n-1)x)+10^p\varphi^p_q(x,nx)+\varphi^p_q(2x,2x)\\&+4^p\varphi^p_q(2x,x)+n^{2p}\varphi^p_q(x,3x)+2^p(3n^2-1)^p\varphi^p_q(x,2x)+(17n^2-8)^p\varphi^p_q(x,x)\\&+\frac{n^{2p}}{(n^2-1)^p}(\varphi^p_q(0,(n+1)x)+\varphi^p_q(0,(n-3)x)+10^p\varphi^p_q(0,(n-1)x)+4^p\varphi^p_q(0,nx)\\&+4^p\varphi^p_q(0,(n-2)x))+\frac{(n^4+1)^p}{(n^2-1)^p}\varphi^p_q(0,2x)+\frac{(2(3n^4-n^2+2))^p}{(n^2-1)^p}\varphi^p_q(0,x)]\hspace
{1.1cm}(3.33)\end{align*}for all $x\in X.$ Therefore, it follows
from (3.2) and (3.33) that
$$\sum^{\infty}_{i=1} 4^{ip} {\psi_{q}}^p(\frac{x}{2^i})<\infty\hspace
{9.4cm}(3.34)\hspace {.2cm}$$ for all $x\in X.$ Thus, we conclude
from (3.32) and (3.34) that the sequence
$\{4^{m}g(\frac{x}{2^m})\}$ is a Cauchy sequence for all $x\in X.$
Since $Y$ is complete, then, the sequence
$\{4^{m}g(\frac{x}{2^m})\}$ converges for all $x\in X.$ So one can
define the function $Q:X\rightarrow Y$ by $$Q(x)=\lim_{m \to
\infty}4^mg(\frac{x}{2^m}) \hspace {9.0cm}(3.35)\hspace {.2cm}$$
for all $x\in X.$ Letting $k=0$ and passing the limit
$m\rightarrow\infty$ in (3.32), we get
$$\|g(x)-Q(x)\|_Y^p\leq M^{8p}\sum^{\infty}_{i=0}4^{ip}{\psi_{q}}^p
(\frac{x}{2^{i+1}})
=\frac{M^{8p}}{4^p}\sum^{\infty}_{i=1}4^{ip}{\psi_{q}}^p(\frac{x}{2^i})
\hspace {3.4cm}(3.36)\hspace {.2cm}$$ for all $x\in X.$ Therefore,
(3.5) follows from (3.2) and (3.36). Now we show that $Q$ is
quadratic. It follows from (3.1), (3.31) and (3.35) that
\begin{align*}
\|Q(2x)-4Q(x)\|_Y &=\lim_{m \to \infty}\|4^m
g(\frac{x}{2^{m-1}})-4^{m+1}g(\frac{x}{2^m})\|_Y\\
&= 4 \lim_{m \to \infty} \|4^{m-1}g(\frac{x}{4^{m-1}})-4^m
g(\frac{x}{2^m})\|_Y \\
& \leq M^{11}\lim_{m \to \infty} 4^m \psi_{q}(\frac{x}{2^m})=0
\hspace {6cm}\end{align*} for all $x \in X.$ So
$$Q(2x)=4Q(x)\hspace {10.25cm}(3.37)$$ for all $x \in X.$ On the
other hand, it follows from (3.1), (3.3), (3.4) and (3.35) that
\begin{align*}
\|\triangle Q(x,y)\|_Y&=\lim_{m \to \infty} 4^m \|\triangle
g(\frac{x}{2^{m}},\frac{y}{2^{m}})\|_Y=\lim_{m \to \infty} 4^m
\|\triangle f(\frac{x}{2^{m-1}},\frac{y}{2^{m-1}})-16
\triangle f(\frac{x}{2^m},\frac{y}{2^m})\|_Y\\
&\leq M\lim_{m \to \infty} 4^m\{\|\triangle
f(\frac{x}{2^{m-1}},\frac{y}{2^{m-1}})\|_Y+16
\|\triangle f(\frac{x}{2^m},\frac{y}{2^m})\|_Y\}\\
&\leq M\lim_{m \to \infty}
4^m\{\varphi_q(\frac{x}{2^{m-1}},\frac{y}{2^{m-1}})+16
\varphi_q(\frac{x}{2^m},\frac{y}{2^m})\}=0
\end{align*}
for all $x,y \in X.$ Hence the function $Q$ satisfies (1.5). By
Lemma 2.1, the function $x \rightsquigarrow Q(2x)-4Q(x)$ is
quadratic. Hence, (3.37) implies that the function $Q$ is quadratic.

It remains  to show that $Q$ is unique.Suppose that there exists
another quadratic function $Q^{'}:X \to Y$ witch satisfies (1.5)
and (3.5).Since $Q^{'}(\frac{x}{2^m})=\frac{1}{4^m}Q^{'}(x)$ and
$Q(\frac{x}{2^m})=\frac{1}{4^m}Q(x)$ for all $x \in X$, we
conclude from (3.5) that
$$\|Q(x)-Q^{'}(x)\|^p_Y=\lim_{m \to \infty}4^{mp}{\|g(\frac{x}{2^m})-Q^{'}(\frac{x}{2^m})\|_Y}^p
\leq \frac{M^{8p}}{4^p}\lim_{m \to
\infty}4^{mp}\widetilde{\psi}_q(\frac{x}{2^m})\eqno\hspace{0.9cm}(3.38)$$
for all $x \in X.$On the other hand, since
$$\lim_{m \to\infty}4^{mp}\sum_{i=1}^{\infty}4^{ip}{\varphi_q}^p(\frac{x}{2^{m+i}},\frac{y}{2^{m+i}})
=\lim_{m\to\infty}\sum_{i=m+1}^{\infty}4^{ip}{\varphi_q}^p
(\frac{x}{2^i},\frac{y}{2^i})=0\eqno\hspace{2cm}$$ for all $x\in X$
and for all $y\in
\{x,2x,3x,nx,(n+1)x,(n-1)x,(n+2)x,(n-2)x,(n-3)x\},$ therefore
$$\lim_{m \to \infty}4^{mp}\widetilde{\psi}_q(\frac{x}{2^m})=0\hspace{9.1cm}(3.39)$$
for all $x \in X$.By using (3.39) in (3.38), we get $Q=Q^{'}.$\\
\end{proof}

\begin{thm}\label{t2}
Let $\varphi_q:X\times X\rightarrow [0,\infty)$ be a function such
that
$$\lim_{m\rightarrow\infty}\frac{1}{4^m}
\varphi_q(2^{m}x,2^{m}y)=0 \eqno\hspace {2.2cm}$$ for all $x,y\in X$
and
$$\sum^{\infty}_{i=0} \frac{1}{4^{pi}}
{\varphi_q}^p(2^{i}x,2^{i}y)<\infty \eqno \hspace {2cm}$$ for all
$x\in X$ and for all $y\in
\{x,2x,3x,nx,(n+1)x,(n-1)x,(n+2)x,(n-2)x,(n-3)x\}.$ Suppose that a
function $f:X\rightarrow Y$ with $f(0)=0$ satisfies the inequality
$$\|\bigtriangleup f(x,y)\|_Y \leq\varphi_q(x,y)\eqno \hspace {4.5cm}$$
for all $x,y\in X.$ Then the limit
$$Q(x):=\lim_{m\rightarrow\infty} \frac{1}{4^{m}}
[f(2^{m+1}x)-16f(2^{m}x)] \eqno \hspace {2.4cm}$$ exists for all
$x\in X$ and $Q:X\rightarrow Y$ is a unique quadratic function
satisfying
$$\|f(2x)-16f(x)-Q(x)\|_Y \leq\frac{M^{8}}{4}[\widetilde{\psi}_q(x)]^\frac{1}{p}\eqno\hspace {2cm}$$
for all $x\in X,$ where
\begin{align*}
\widetilde{\psi}_q(x):&=\sum^{\infty}_{i=0}
\frac{1}{4^{pi}}~\textbf{\{}\frac{1}{n^{2p}(n^2-1)^p}
~[~\varphi^p_q(2^{i}x,2^{i}(n+2)x)+\varphi^p_q(2^{i}x,2^{i}(n-2)x)\\&+4^p\varphi^p_q(2^{i}x,2^{i}(n+1)x)+4^p\varphi^p_q(2^{i}x,2^{i}(n-1)x)+10^p\varphi^p_q(2^{i}x,2^{i}nx)+\varphi^p_q(2^{i}2x,2^{i}2x)\\&+4^p\varphi^p_q(2^{i}2x,2^{i}x)+n^{2p}\varphi^p_q(2^{i}x,2^{i}3x)+2^p(3n^2-1)^p\varphi^p_q(2^{i}x,2^{i}2x)\\&+(17n^2-8)^p\varphi^p_q(2^{i}x,2^{i}x)+\frac{n^{2p}}{(n^2-1)^p}(\varphi^p_q(0,2^{i}(n+1)x)+\varphi^p_q(0,2^{i}(n-3)x)\\&+10^p\varphi^p_q(0,2^{i}(n-1)x)+4^p\varphi^p_q(0,2^{i}nx)+4^p\varphi^p_q(0,2^{i}(n-2)x))\\&+\frac{(n^4+1)^p}{(n^2-1)^p}\varphi^p_q(0,2^{i}2x)+\frac{(2(3n^4-n^2+2))^p}{(n^2-1)^p}\varphi^p_q(0,2^{i}x)~]~
\textbf{\}}. \hspace{2.7cm}\end{align*}
\end{thm}
\begin{proof}
The proof is similar to the proof of Theorem 3.2.
\end{proof}
\begin{cor}\label{t2}
Let $\theta, r, s$ be non-negative real numbers such that $r,s> 2$
or $s<2$. Suppose that a function $f:X\rightarrow Y$ with $f(0)=0$
satisfies the inequality
$$\|\triangle f(x,y)\|_Y \leq\left\{%
\begin{array}{ll}
    \theta, & \hbox{r =s =0;} \\
    \theta\|x\|_X^r, & \hbox{r $>$ 0, s=0;} \\
    \theta\|y\|_X^s, & \hbox{r=0, s $>$ 0;} \\
    \theta(\|x\|_X^r+\|y\|_X^s), & \hbox{r, s $>$ 0.} \\
\end{array}%
\right.\eqno\hspace{3.4cm}(3.40)$$ for all $x, y \in X.$ Then there
exists a unique quadratic function $Q:X\rightarrow Y$ satisfying
$$\| f(2x)-16f(x)-Q(x)\|_Y \leq \frac{M^{8} \theta}{n^2(n^2-1)}
\left\{%
\begin{array}{ll}
    \delta_{q}, & \hbox{r =s =0;} \\
    \alpha_{q}(x), & \hbox{r $>$ 0, s=0;} \\
    \beta_{q}(x), & \hbox{r=0, s $>$ 0;} \\
    (\alpha^p_{q}(x)+\beta^p_{q}(x))^{\frac{1}{p}} & \hbox{r, s $>$ 0.} \\
\end{array}%
\right.\eqno\hspace{2.2cm}$$ for all $x \in X,$ where
\begin{align*}\delta_{q}&=\textbf{\{}~\frac{1}{{4^p-1}(n^2-1)^p}[(6n^2-2)^{p}(n^2-1)^p+(17n^2-8)^{p}(n^2-1)^p+(6n^4-2n^2+4)^p\\&+n^{2p}(2+10^p+2*4^p)+(n^4+1)^p+n^{2p}(n^2-1)^p+3*4^{p}(n^2-1)^p+10^{p}(n^2-1)^p\\&+3(n^2-1)^p]
~\textbf{\}}^\frac{1}{p},\hspace{5cm}\end{align*}
\begin{align*}\alpha_{q}(x)=\textbf{\{}~\frac{4^{p}(2+2^{rp})+10^p+(6n^2-2)^{p}+(17n^2-8)^{p}+2^{rp}+n^{2p}}{|4^p-2^{rp}|}
~\textbf{\}}^\frac{1}{p}\|x\|_X^r \hspace{5cm}\end{align*}and
\begin{align*}\beta_{q}(x)&=\textbf{\{}~\frac{1}{(n^2-1)^{p}|4^p-2^{sp}|}[2^{sp}(6n^2-2)^{p}(n^2-1)^p+(17n^2-8)^{p}(n^2-1)^p\\&+(6n^4-2n^2+4)^p+n^{2p}((n+1)^{sp}+(n-3)^{sp}+10^{p}(n-1)^{sp}\\&+4^{p}n^{sp}+4^{p}(n-2)^{sp})+2^{sp}(n^4+1)^p+3^{sp}n^{2p}(n^2-1)^p+4^{p}(n^2-1)^p\\&+(n+2)^{sp}(n^2-1)^p+(n-2)^{sp}(n^2-1)^p+4^{p}(n+1)^{sp}(n^2-1)^p\\&+4^{p}(n-1)^{sp}(n^2-1)^p+10^{p}n^{sp}(n^2-1)^p]
~\textbf{\}}^\frac{1}{p}\|x\|_X^s .\hspace{7.3cm}\end{align*}
\end{cor}
\begin{proof}
In Theorem 3.2, putting  $\varphi_{q}(x,
y):=\theta(\|x\|_X^r+\|y\|_X^s)$ for all $x, y \in X.$
\end{proof}
\begin{cor}\label{t2}
Let $\theta\geq0$ and $r, s>0$ be non-negative real numbers such
that $\lambda:=r+s\neq2$. Suppose that a function $f:X\rightarrow Y$
with $f(0)=0$ satisfies the inequality
$$\|\triangle f(x,y)\|_Y \leq\theta\|x\|_X^r\|y\|_X^s ,\eqno \hspace {4cm} (3.41)$$
for all $x, y \in X.$ Then there exists a unique quadratic function
 $Q:X\rightarrow Y$ satisfying
\begin{align*}\|f(2x)-16f(x)-Q(x)\|_Y &\leq\frac{M^{8}\theta}{n^2(n^2-1)}
\textbf{\{}~\frac{1}{|4^p-2^{\lambda
p}|}[(n+2)^{sp}+(n-2)^{sp}+4^{p}(n+1)^{sp}\\&+4^{p}(n-1)^{sp}+10^{p}n^{sp}+2^{(r+s)p}+4^{p}2^{rp}+n^{2p}3^{sp}\\&+2^{sp}(6n^2-2)^{p}+(17n^2-8)^{p}]
~\textbf{\}}^\frac{1}{p}~\|x\|_X^\lambda
\hspace{7.3cm}\end{align*} for all $x \in X.$
\end{cor}
\begin{proof}
In Theorem 3.2 putting $\varphi_{q}(x, y):=\theta\|x\|_X^r
\|y\|_X^s$ for all $x, y \in X.$
\end{proof}
\begin{thm}\label{t2}
Let $\varphi_t:X\times X\rightarrow [0,\infty)$ be a function such
that
$$\lim_{m\rightarrow\infty} 16^{m}
\varphi_t(\frac{x}{2^{m}},\frac{y}{2^{m}})=0 \eqno\hspace
{2cm}(3.42)$$ for all $x,y\in X$ and
$$\sum^{\infty}_{i=1} 16^{pi}
{\varphi_t}^p(\frac{x}{2^{i}},\frac{y}{2^{i}})<\infty \eqno \hspace
{2cm}(3.43)$$ for all $x\in X$ and for all $y\in
\{x,2x,3x,nx,(n+1)x,(n-1)x,(n+2)x,(n-2)x,(n-3)x\}.$ Suppose that a
function $f:X\rightarrow Y$ with $f(0)=0$ satisfies the inequality
$$\|\triangle f(x,y)\|_Y \leq\varphi_t(x,y)\eqno \hspace {4.5cm} (3.44)$$
for all $x,y\in X.$ Then the limit
$$T(x):=\lim_{m\rightarrow\infty} 16^{m}
[f(\frac{x}{2^{m-1}})-4f(\frac{x}{2^{m}})] \eqno \hspace
{2.4cm}(3.45)$$ exists for all $x\in X$ and $T:X\rightarrow Y$ is a
unique quartic function satisfying
$$\|f(2x)-4f(x)-T(x)\|_Y \leq\frac{M^{8}}{16}[\widetilde{\psi}_t(x)]^\frac{1}{p}
\eqno\hspace {2cm}(3.46)$$ for all $x\in X,$ where
\begin{align*}
\widetilde{\psi}_t(x):&=\sum^{\infty}_{i=1}
16^{pi}~\textbf{\{}\frac{1}{n^{2p}(n^2-1)^p}
~[~\varphi^p_t(\frac{x}{2^i},\frac{(n+2)x}{2^i})+\varphi^p_t(\frac{x}{2^i},\frac{(n-2)x}{2^i})\\&+4^p\varphi^p_t(\frac{x}{2^i},\frac{(n+1)x}{2^i})+4^p\varphi^p_t(\frac{x}{2^i},\frac{(n-1)x}{2^i})+10^p\varphi^p_t(\frac{x}{2^i},\frac{nx}{2^i})+\varphi^p_t(\frac{2x}{2^i},\frac{2x}{2^i})\\&+4^p\varphi^p_t(\frac{2x}{2^i},\frac{x}{2^i})+n^{2p}\varphi^p_t(\frac{x}{2^i},\frac{3x}{2^i})+2^p(3n^2-1)^p\varphi^p_t(\frac{x}{2^i},\frac{2x}{2^i})\\&+(17n^2-8)^p\varphi^p_t(\frac{x}{2^i},\frac{x}{2^i})+\frac{n^{2p}}{(n^2-1)^p}(\varphi^p_t(0,\frac{x(n+1)x}{2^i})+\varphi^p_t(0,\frac{(n-3)x}{2^i})\\&+10^p\varphi^p_t(0,\frac{(n-1)x}{2^i})+4^p\varphi^p_t(0,\frac{nx}{2^i})+4^p\varphi^p_t(0,\frac{(n-2)x}{2^i}))\\&+\frac{(n^4+1)^p}{(n^2-1)^p}\varphi^p_t(0,\frac{2x}{2^i})+\frac{(2(3n^4-n^2+2))^p}{(n^2-1)^p}\varphi^p_t(0,\frac{x}{2^i})~]~
\textbf{\}}. \hspace{3.1cm}(3.47)\end{align*}
\end{thm}
\begin{proof}
Similar to the proof Theorem 3.2, we have
$$\|f(4x)-20f(2x)+64f(x)\|\leq M^{8}\psi_{t}(x), \eqno \hspace {6cm}(3.48)$$
for all $x \in X,$ where
\begin{align*}\psi_{t}(x)&=
\frac{1}{n^{2}(n^2-1)}[\varphi_t(x,(n+2)x)+\varphi_t(x,(n-2)x)\\&+4\varphi_t(x,(n+1)x)+4\varphi_t(x,(n-1)x)+10\varphi_t(x,nx)+\varphi_t(2x,2x)\\&+4\varphi_t(2x,x)+n^2\varphi_t(x,3x)+2(3n^2-1)\varphi_t(x,2x)+(17n^2-8)\varphi_t(x,x)\\&+\frac{n^2}{n^2-1}(\varphi_t(0,(n+1)x)+\varphi_t(0,(n-3)x)+10\varphi_t(0,(n-1)x)+4\varphi_t(0,nx)\\&+4\varphi_t(0,(n-2)x))+\frac{n^4+1}{n^2-1}\varphi_t(0,2x)+\frac{2(3n^4-n^2+2)}{n^2-1}\varphi_t(0,x)].\hspace
{2cm}(3.49)\end{align*}
 Let $h:X\to Y$ be a function defined by $h(x):=f(2x)-4f(x)$. Then, we
conclude that
$$\|h(2x)-16h(x)\|\leq M^{8}\psi_{t}(x) \eqno\hspace {7.5cm}(3.50)$$ for all $x \in X.$
If we replace $x$ in (3.50) by $\frac{x}{2^{m+1}}$ and multiply both
sides of (3.50) by $16^m,$ we get
$$\|16^{m+1}h(\frac{x}{2^{m+1}})-16^mh(\frac{x}{2^m})\|_Y\leq M^{8} 16^m \psi_{t}(\frac{x}{2^{m+1}})
 \eqno\hspace {2.5cm}(3.51)$$
for all $x\in X$ and all non-negative integers $m$. Since $Y$ is a
p-Banach space, therefore, inequality (3.51)  gives
\begin{align*}
\|16^{m+1}h(\frac{x}{2^{m+1}})-16^k h(\frac{x}{2^k})\|_Y^p&\leq
\sum^{m}_{i=k}\|16^{i+1}h(\frac{x}{2^{i+1}})-16^i
h(\frac{x}{2^i})\|_Y^p\\
&\leq M^{8p} \sum^{m}_{i=k} 16^{pi} {\psi_{t}}^p
(\frac{x}{2^{i+1}}) \hspace {3.9cm}(3.52)
\end{align*}
for all non-negative integers $m$ and $k$ with $m\geq k$ and all
$x\in X.$ Since $0< p\leq1$, then by Lemma 3.1, we conclude from
(3.49) that

\begin{align*}\psi_{t}^p(x)&\leq
\frac{1}{n^{2p}(n^2-1)^p}~[\varphi^p_t(x,(n+2)x)+\varphi^p_t(x,(n-2)x)\\&+4^p\varphi^p_t(x,(n+1)x)+4^p\varphi^p_t(x,(n-1)x)+10^p\varphi^p_t(x,nx)+\varphi^p_t(2x,2x)\\&+4^p\varphi^p_t(2x,x)+n^{2p}\varphi^p_t(x,3x)+2^p(3n^2-1)^p\varphi^p_t(x,2x)+(17n^2-8)^p\varphi^p_t(x,x)\\&+\frac{n^{2p}}{(n^2-1)^p}(\varphi^p_t(0,(n+1)x)+\varphi^p_t(0,(n-3)x)+10^p\varphi^p_t(0,(n-1)x)+4^p\varphi^p_t(0,nx)\\&+4^p\varphi^p_t(0,(n-2)x))+\frac{(n^4+1)^p}{(n^2-1)^p}\varphi^p_t(0,2x)+\frac{(2(3n^4-n^2+2))^p}{(n^2-1)^p}\varphi^p_t(0,x)~],
\hspace {0.7cm}(3.53)\end{align*}
 for all $x\in X.$ Therefore, it follows from (3.42) and (3.52) that
$$\sum^{\infty}_{i=1} 16^{pi} {\psi_{t}}^p(\frac{x}{2^i})<\infty\hspace
{9.15cm}(3.54)\hspace {.2cm}$$ for all $x\in X.$ Thus, we conclude
from (3.52) and (3.54) that the sequence
$\{16^{m}h(\frac{x}{2^m})\}$ is a Cauchy sequence for all $x\in X.$
Since $Y$ is complete, the sequence $\{16^{m}h(\frac{x}{2^m})\}$
converges for all $x\in X.$ So one can define the function
$T:X\rightarrow Y$ by $$T(x)=\lim_{m \to \infty}16^mh(\frac{x}{2^m})
\hspace {8.8cm}(3.55)\hspace {.2cm}$$ for all $x\in X.$ Letting
$k=0$ and passing the limit $m\rightarrow\infty$ in (3.52), we get
$$\|h(x)-T(x)\|_Y^p\leq M^{8p}\sum^{\infty}_{i=0}16^{pi}{\psi_{t}}^p
(\frac{x}{2^{i+1}})
=\frac{M^{11p}}{16^p}\sum^{\infty}_{i=1}16^{pi}{\psi_{t}}^p(\frac{x}{2^i})
\hspace {2.8cm}(3.56)\hspace {.2cm}$$ for all $x\in X.$ Therefore
(3.45) follows from (3.43) and (3.55). Now we show that $T$ is
quartic. According to (3.42), (3.51) and (3.55), it follows that
\begin{align*}
\|T(2x)-16T(x)\|_Y &=\lim_{m \to \infty}\|16^m
h(\frac{x}{2^{m-1}})-16^{m+1}h(\frac{x}{2^m})\|_Y\\
&= 16 \lim_{m \to \infty} \|16^{m-1}h(\frac{x}{16^{m-1}})-16^m
h(\frac{x}{2^m})\|_Y \\
& \leq M^{8}\lim_{m \to \infty} 16^m \psi_{t}(\frac{x}{2^m})=0
\hspace {6cm}\end{align*} for all $x \in X.$ So
$$T(2x)=16T(x)\hspace {9.8cm}(3.57)$$ for all $x \in X.$ On the
other hand, by (3.44), (3.54) and (3.55), we lead to
\begin{align*}
\|\triangle T(x,y)\|_Y&=\lim_{m \to \infty} 16^m \|\triangle
h(\frac{x}{2^{m}},\frac{y}{2^{m}})\|_Y=\lim_{m \to \infty} 16^m
\|\triangle f(\frac{x}{2^{m-1}},\frac{y}{2^{m-1}})-4
\triangle f(\frac{x}{2^m},\frac{y}{2^m})\|_Y\\
&\leq M\lim_{m \to \infty} 16^m\{\|\triangle
f(\frac{x}{2^{m-1}},\frac{y}{2^{m-1}})\|_Y+4
\|\triangle f(\frac{x}{2^m},\frac{y}{2^m})\|_Y\}\\
&\leq M\lim_{m \to \infty}
16^m\{\varphi_t(\frac{x}{2^{m-1}},\frac{y}{2^{m-1}})+4
\varphi_t(\frac{x}{2^m},\frac{y}{2^m})\}=0
\end{align*}
for all $x,y \in X.$ Hence, the function $T$ satisfies (1.5). By
Lemma 2.1, the function $x \rightsquigarrow T(2x)-16T(x)$ is
quartic. Therefore (3.57) implies that the function $T$ is quartic.\\

To prove the uniqueness property of $T,$ let $T^{'}:X \to Y$ be
another quartic function satisfying (3.46). Since $$\lim_{m \to
\infty}16^{mp}\sum_{i=1}^{\infty}16^{pi}{\varphi_t}^p(\frac{x}{2^{m+i}},\frac{x}{2^{m+i}})
=\lim_{m\to\infty}\sum_{i=m+1}^{\infty}16^{pi}{\varphi_t}^p
(\frac{x}{2^i},\frac{x}{2^i})=0\eqno\hspace{2cm}$$ for all $x\in X$
and for all $y\in
\{x,2x,3x,nx,(n+1)x,(n-1)x,(n+2)x,(n-2)x,(n-3)x\},$ then
$$\lim_{m \to \infty}16^{mp}\widetilde{\psi}_t(\frac{x}{2^m})=0\hspace{9cm}(3.58)$$
for all $x \in X$. It follows from (3.46) and (3.58) that
$$\|T(x)-T^{'}(x)\|_Y=\lim_{m \to \infty}16^{mp}{\|h(\frac{x}{2^m})-T^{'}(\frac{x}{2^m})\|_Y}^p
\leq \frac{M^{8p}}{16^p}\lim_{m \to
\infty}16^{mp}\widetilde{\psi}_t(\frac{x}{2^m})=0$$
for all $x \in X.$ So $T=T^{'}.$\\
\end{proof}
\begin{thm}\label{t2}
Let $\varphi_t:X\times X\rightarrow [0,\infty)$ be a function such
that
$$\lim_{m\rightarrow\infty} \frac{1}{16^{m}}
\varphi_t(2^{m}x,2^{m}y)=0 \eqno\hspace {2cm}$$ for all $x,y\in X$
and
$$\sum^{\infty}_{i=0} \frac{1}{16^{pi}}
{\varphi_t}^p(2^{i}x,2^{i}y)<\infty \eqno \hspace {2cm}$$ for all
$x\in X$ and for all $y\in
\{x,2x,3x,nx,(n+1)x,(n-1)x,(n+2)x,(n-2)x,(n-3)x\}.$ Suppose that a
function $f:X\rightarrow Y$ with $f(0)=0$ satisfies the inequality
$$\|\triangle f(x,y)\|_Y \leq\varphi_t(x,y)\eqno \hspace {4.5cm}$$
for all $x,y\in X.$ Then the limit
$$T(x):=\lim_{m\rightarrow\infty} \frac{1}{16^{m}}
[f(2^{m+1}x)-4f(2^{m}x)] \eqno \hspace {2.4cm}$$ exists for all
$x\in X$ and $T:X\rightarrow Y$ is a unique quartic function
satisfying
$$\|f(2x)-4f(x)-T(x)\|_Y \leq\frac{M^{8}}{16}[\widetilde{\psi}_t(x)]^\frac{1}{p}
\eqno\hspace {2cm}$$ for all $x\in X,$ where
\begin{align*}
\widetilde{\psi}_t(x):&=\sum^{\infty}_{i=0}
\frac{1}{16^{pi}}~\textbf{\{}\frac{1}{n^{2p}(n^2-1)^p}
~[~\varphi^p_t(2^{i}x,2^{i}(n+2)x)+\varphi^p_t(2^{i}x,2^{i}(n-2)x)\\&+4^p\varphi^p_t(2^{i}x,2^{i}(n+1)x)+4^p\varphi^p_t(2^{i}x,2^{i}(n-1)x)+10^p\varphi^p_t(2^{i}x,2^{i}nx)+\varphi^p_t(2^{i}2x,2^{i}2x)\\&+4^p\varphi^p_t(2^{i}2x,2^{i}x)+n^{2p}\varphi^p_t(2^{i}x,2^{i}3x)+2^p(3n^2-1)^p\varphi^p_t(2^{i}x,2^{i}2x)\\&+(17n^2-8)^p\varphi^p_t(2^{i}x,2^{i}x)+\frac{n^{2p}}{(n^2-1)^p}(\varphi^p_t(0,2^{i}(n+1)x)+\varphi^p_t(0,2^{i}(n-3)x)\\&+10^p\varphi^p_t(0,2^{i}(n-1)x)+4^p\varphi^p_t(0,2^{i}nx)+4^p\varphi^p_t(0,2^{i}(n-2)x))\\&+\frac{(n^4+1)^p}{(n^2-1)^p}\varphi^p_t(0,2^{i}2x)+\frac{(2(3n^4-n^2+2))^p}{(n^2-1)^p}\varphi^p_t(0,2^{i}x)~]~
\textbf{\}}. \hspace{2.7cm}\end{align*}
\end{thm}
\begin{proof}
The proof is similar to the proof of Theorem 3.6.
\end{proof}
\begin{cor}\label{t2}
Let $\theta, r, s$ be non-negative real numbers such that $r,s> 4$
or $0\leq r,s<4$. Suppose that a function $f:X\rightarrow Y$ with
$f(0)=0$ satisfies the inequality (3.40) for all $x, y \in X.$ Then
there exists a unique quartic function $T:X\rightarrow Y$ satisfying
$$\| f(2x)-4f(x)-T(x)\|_Y \leq \frac{M^{8} \theta}{n^2(n^2-1)}
\left\{%
\begin{array}{ll}
    \delta_{t}, & \hbox{r =s =0;} \\
    \alpha_{t}(x), & \hbox{r $>$ 0, s=0;} \\
    \beta_{t}(x), & \hbox{r=0, s $>$ 0;} \\
    (\alpha^p_{t}(x)+\beta^p_{t}(x))^{\frac{1}{p}}, & \hbox{r, s $>$ 0.} \\
\end{array}%
\right.\eqno\hspace{2.2cm}$$ for all $x \in X,$ where
\begin{align*}\delta_{t}&=\textbf{\{}~\frac{1}{(16^p-1)(n^2-1)^p}[(6n^2-2)^{p}(n^2-1)^p+(17n^2-8)^{p}(n^2-1)^p+(6n^4-2n^2+4)^p\\&+n^{2p}(2+10^p+2*4^p)+(n^4+1)^p+n^{2p}(n^2-1)^p+3*4^{p}(n^2-1)^p+10^{p}(n^2-1)^p\\&+3(n^2-1)^p]
~\textbf{\}}^\frac{1}{p},\hspace{5cm}\end{align*}
\begin{align*}\alpha_{t}(x)=\textbf{\{}~\frac{4^{p}(2+2^{rp})+10^{p}+(6n^2-2)^{p}+(17n^2-8)^{p}+2^{rp}+n^{2p}}{|16^p-2^{rp}|}
~\textbf{\}}^\frac{1}{p}\|x\|_X^r \hspace{5cm}\end{align*}and
\begin{align*}\beta_{t}(x)&=\textbf{\{}~\frac{1}{(n^2-1)^{p}|16^p-2^{sp}|}[2^{sp}(6n^2-2)^{p}(n^2-1)^p+(17n^2-8)^{p}(n^2-1)^p\\&+(6n^4-2n^2+4)^p+n^{2p}((n+1)^{sp}+(n-3)^{sp}+10^{p}(n-1)^{sp}\\&+10^{p}n^{sp}+4^{p}(n-2)^{sp})+2^{sp}(n^4+1)^p+3^{sp}n^{2p}(n^2-1)^p+4^{p}(n^2-1)^p\\&+(n+2)^{sp}(n^2-1)^p+(n-2)^{sp}(n^2-1)^p+4^{p}(n+1)^{sp}(n^2-1)^p\\&+4^{p}(n-1)^{sp}(n^2-1)^p+4^{p}n^{sp}(n^2-1)^p]
~\textbf{\}}^\frac{1}{p}\|x\|_X^s .\hspace{7.3cm}\end{align*}
\end{cor}
\begin{proof}
In Theorem 3.6,  putting  $\varphi_{t}(x,
y):=\theta(\|x\|_X^r+\|y\|_X^s)$ for all $x, y \in X.$
\end{proof}

\begin{cor}\label{t2}
Let $\theta\geq0$ and $r, s>0$ be non-negative real numbers such
that $\lambda:=r+s\neq4$. Suppose that a function $f:X\rightarrow Y$
with $f(0)=0$  satisfies the inequality (3.41) for all $x, y \in X.$
Then there exists a unique quartic function
 $T:X\rightarrow Y$ satisfying
\begin{align*}\|f(2x)-4f(x)-T(x)\|_Y &\leq\frac{M^{8}\theta}{n^2(n^2-1)}
\textbf{\{}~\frac{1}{|16^p-2^{\lambda
p}|}[(n+2)^{sp}+(n-2)^{sp}+4^{p}(n+1)^{sp}\\&+4^{p}(n-1)^{sp}+10^{p}n^{sp}+2^{(r+s)p}+4^{p}2^{rp}+n^{2p}3^{sp}\\&+2^{sp}(6n^2-2)^{p}+(17n^2-8)^{p}]
~\textbf{\}}^\frac{1}{p}~\|x\|_X^\lambda\hspace{5cm}\end{align*}
for all $x \in X.$
\end{cor}
\begin{proof}
In Theorem 3.6, putting $\varphi_{t}(x, y):=\theta\|x\|_X^r
\|y\|_X^s$ for all $x, y \in X.$
\end{proof}
\begin{thm}\label{t2}
Let  $\varphi:X\times X\rightarrow [0, \infty)$ be a function such
that
$$\lim_{m\rightarrow\infty}4^{m}
\varphi(\frac{x}{2^{m}},\frac{y}{2^{m}})
=0=\lim_{m\rightarrow\infty} \frac{1}{16^{m}}
\varphi(2^{m}x,2^{m}y)\eqno \hspace {4cm}(3.59)$$ for all $x,y\in X$
and
$$\sum^{\infty}_{i=1}4^{pi}
\varphi^p(\frac{x}{2^{i}},\frac{y}{2^{i}})<\infty \eqno \hspace
{5cm}$$and
$$\sum^{\infty}_{i=0}
\frac{1}{16^{pi}} \varphi^p(2^{i}x,2^{i}y)<\infty\eqno \hspace
{5cm}$$for all $x\in X$ and for all $y\in
\{x,2x,3x,nx,(n+1)x,(n-1)x,(n+2)x,(n-2)x,(n-3)x\}$. Suppose that a
function $f:X\rightarrow Y$ with $f(0)=0$  satisfies the inequality
$$\|\triangle f(x,y)\|_Y \leq\varphi(x,y),\eqno\hspace{6.5 cm} (3.60)$$
for all $x,y\in X.$ Then there exist a unique quadratic function
$Q:X \to Y$ and a unique quartic function $T:X \to Y$ such that
$$\|f(x)-Q(x)-T(x)\|_Y\leq \frac{M^{9}}{192}~( 4[\widetilde{\psi}_q(x)]^{\frac{1}{p}}
+[\widetilde{\psi}_t(x)]^{\frac{1}{p}} )\eqno\hspace{2 cm}(3.61)$$
for all $x \in X,$ where $\widetilde{\psi}_q(x)$ and
 $\widetilde{\psi}_t(x)$ have been defined in Theorems 3.2 and 3.7,
respectively, for all $x \in X.$
\end{thm}
\begin{proof}
By Theorems 3.2 and 3.7, there exist a quadratic function $Q_0:X \to
Y$ and a quartic function $T_0:X \to Y$ such that
$$~\|f(2x)-16f(x)-Q_0(x)\|_Y
\leq\frac{M^{8}}{4}[\widetilde{\psi}_q(x)]^{\frac{1}{p}},
 \hspace{.8cm}\|f(2x)-4f(x)-T_0(x)\|_Y\leq\frac{M^{8}}{16}[\widetilde{\psi}_t(x)]^{\frac{1}{p}}$$
for all $x \in X.$ Therefore, it follows from the last
inequalities that
$$\|f(x)+\frac{1}{12}Q_0(x)-\frac{1}{12}T_0(x)\|_Y \leq
\frac{M^{9}}{192}~( 4[\widetilde{\psi}_q(x)]^{\frac{1}{p}}
+[\widetilde{\psi}_t(x)]^{\frac{1}{p}} )\eqno\hspace{4.4 cm}$$ for
all $x \in X.$ So we obtain $(3.61)$ by letting
$Q(x)=-\frac{1}{12}Q_0(x)$ and $T(x)=\frac{1}{12}T_0(x)$ for all
$x \in X.$

To prove the uniqueness property of $Q$ and $T,$ we first show the
uniqueness property for $Q_0$ and $T_0$ and then we conclude the
uniqueness property of $Q$ and $T.$ Let $Q_1,T_1:X \to Y$ be
another quadratic and quartic functions satisfying (3.61) and let
$Q_2=\frac{1}{12}Q_0$, $T_2=\frac{1}{12}T_0$, $Q_3=Q_{2}-Q_1$ and
$T_3=T_{2}-T_1.$ So
\begin{align*}
\|Q_3(x)-T_3(x)\|_Y &\leq M\{\|f(x)-Q_2(x)-T_2(x)\|_Y
+\|f(x)-Q_1(x)-T_1(x)\|_Y \}\\&\leq\frac{M^{10}}{96}~(
4[\widetilde{\psi}_q(x)]^{\frac{1}{p}}
+[\widetilde{\psi}_t(x)]^{\frac{1}{p}} ) \hspace{5.1cm} (3.62)
\end{align*}
for all $x \in x.$ Since
$$\lim_{m \to \infty} 4^{mp} \widetilde{\psi}_q(\frac{x}{2^m})=
\lim_{m \to
\infty}\frac{1}{16^{mp}}\widetilde{\psi}_t(2^{m}x)=0\eqno\hspace{5cm}$$
for all $x \in X,$ then (3.62) implies that $\lim_{m \to
\infty}\|4^{m}Q_3(\frac{x}{2^m})+\frac{1}{16^m}T_3(2^{m}x)\|_Y=0$
for all $x \in X.$ Thus, $T_3=Q_3.$ But $T_3$ is only a quartic
function and $Q_3$ is only a quadratic function. Therefore, we
should have $T_3=Q_3=0$ and this complete the uniqueness property
of $Q$ and $T.$ The other results proved similarly.
\end{proof}

\begin{cor}\label{t2}
Let $\theta, r, s$ be non-negative real numbers such that $r,s> 4$
or $2< r,s <4$ or $0\leq r,s<2$. Suppose that a function
$f:X\rightarrow Y$ satisfies the inequality (3.40)
 for all $x, y \in X.$ Then
there exist a unique quadratic function $Q:X\rightarrow Y$ and a
unique quartic function $T:X\rightarrow Y$ such that
$$\| f(x)-Q(x)-T(x)\|_Y \leq \frac{M^{9} \theta}{12n^2(n^2-1)}
\left\{%
\begin{array}{ll}
    \delta_{q}+\delta_{t}, & \hbox{r =s =0;} \\
    \alpha_{q}(x)+\alpha_{t}(x), & \hbox{r $>$ 0, s=0;} \\
    \beta_{q}(x)+\beta_{t}(x), & \hbox{r=0, s $>$ 0;} \\
    (\alpha^p_{q}(x)+\beta^p_{q}(x))^{\frac{1}{p}}+(\alpha^p_{t}(x)+\beta^p_{t}(x))^{\frac{1}{p}}, & \hbox{r, s $>$ 0.} \\
\end{array}%
\right.\eqno\hspace{2cm}$$ for all $x \in X,$ where $\delta_{q},
\delta_{t}, \alpha_{q}(x), \alpha_{t}(x), \beta_{q}(x)$  and
 $\beta_{t}(x)$ are defined as in Corollaries 3.4 and 3.8.
\end{cor}
\begin{cor}\label{t2}
Let $\theta\geq0$ and $r, s>0$ be non-negative real numbers such
that $\lambda:=r+s \in (0,2)\cup(2,4)\cup(4,\infty)$. Suppose that a
function $f:X\rightarrow Y$ satisfies the inequality (3.41) for all
$x, y \in X.$ Then there exist a unique quadratic function
$Q:X\rightarrow Y$ and a unique quartic function $T:X\rightarrow Y$
such that
\begin{align*}\|f(x)-Q(x)-T(x)\|_Y &\leq\frac{M^{9}\theta}{12n^2(n^2-1)}
\textbf{\{}~\frac{1}{|4^p-2^{\lambda
p}|}[(n+2)^{sp}+(n-2)^{sp}+4^{p}(n+1)^{sp}\\&+4^{p}(n-1)^{sp}+10^{p}n^{sp}+2^{(r+s)p}+4^{p}2^{rp}+n^{2p}3^{sp}\\&+2^{sp}(6n^2-2)^{p}+(17n^2-8)^{p}]
~\textbf{\}}^\frac{1}{p}~\|x\|_X^\lambda \hspace{5.1cm}
\end{align*}for all $x \in X.$
\end{cor}

{\small


}
\end{document}